\documentclass{amsart}
\numberwithin{equation}{section}

\newtheorem{thm}{Theorem}[section]
\newtheorem{lem}[thm]{Lemma}

\newtheorem{defin}[thm]{Definition}
\newtheorem{rem}[thm]{Remark}

\begin{document}
\title{Initial-boundary value
and inverse problems for subdiffusion  equations in
$\mathbb{R}^N$}

\author{Ravshan Ashurov}
\author{Raxim Zunnunov}
\address{Institute of Mathematics, Uzbekistan Academy of Science}
\curraddr{Institute of Mathematics, Uzbekistan Academy of Science,
Tashkent, 81 Mirzo Ulugbek str. 100170} \email{ashurovr@gmail.com}

\small

\title[Initial-boundary value
and inverse problems] {Initial-boundary value and inverse problems
for subdiffusion  equations in $\mathbb{R}^N$}

\begin{abstract}

An initial-boundary value problem for a subdiffusion equation with
an elliptic operator $A(D)$ in $\mathbb{R}^N$ is considered.
Uniqueness and existence theorems for a solution of this problem
are proved by the Fourier method. Considering the order of the
Caputo time-fractional derivative as an unknown parameter, the
corresponding inverse problem of determining this order is
studied.  It is proved, that the Fourier transform of the solution
$\hat{u}(\xi, t)$ at a fixed time instance recovers uniquely the
unknown parameter. Further, a similar initial-boundary value
problem is investigated in the case when operator $A(D)$ is
replaced by its power $A^\sigma$. Finally, existence and
uniqueness theorems for the solution of the inverse problem of
determining both the orders of fractional derivatives with respect
to time and the degree of $ \sigma $ are proved.

\end{abstract}

\maketitle

\section{Introduction and main results}

The theory of differential equations with fractional derivatives
has gained significant popularity and importance in the last few
decades, mainly due to its applications in many seemingly distant
fields of science and technology (see, for example, \cite{Mach} -
\cite{Gor}).

One of the most important time-fractional equations is the
subdiffusion equation, which models anomalous or slow diffusion
processes. This equation is a partial integro-differential
equation obtained from the classical heat equation by replacing
the first-order derivative with a time-fractional derivative of
the order $ \rho \in (0, 1) $.

When considering the subdiffusion equation as a model equation in
the analysis of anomalous diffusion processes, the order of the
fractional derivative is often unknown and difficult to measure
directly. To determine this parameter, it is necessary to
investigate the inverse problems of identifying these physical
quantities based on some indirectly observable information about
solutions (see a survey paper  Li, Liu and Yamamoto \cite{LiLiu}).

In this paper, we investigate the existence and uniqueness of
solutions to initial-boundary value problems for subdiffusion
equations with the Caputo derivative and the elliptic operator $ A
(D) $ in $ \mathbb{R}^N $ with constant coefficients. The inverse
problems of determining the order of the fractional derivative
with respect to time and with respect to the spatial variable will
also be investigated.

Let us proceed to a rigorous formulation of the main results of
this article.

\textbf{1.} Let $A(D)=\sum\limits_{|\alpha|=m} a_\alpha D^\alpha$
be a homogeneous symmetric elliptic differential expression of
even order $m=2l$, with constant coefficients, i.e. $A(\xi) >0$,
for all $\xi\neq 0$, where $\alpha=(\alpha_1, \alpha_2, ...,
\alpha_N)$ - multi-index and $D=(D_1, D_2, ..., D_N)$,
$D_j=\frac{\partial}{\partial x_j}$.

The fractional integration in the Riemann - Liouville sense of
order $\rho<0$ has  the form
$$
\partial_t^\rho h(t)=\frac{1}{\Gamma
(-\rho)}\int\limits_0^t\frac{h(\xi)}{(t-\xi)^{\rho+1}} d\xi, \quad
t>0,
$$
provided the right-hand side exists. Here $\Gamma(\rho)$ is
Euler's gamma function. Using this definition one can define the
Caputo fractional derivative of order $\rho$, $0<\rho< 1$, as
$$
D_t^\rho h(t)= \partial_t^{\rho-1} \frac{d}{dt}h(t).
$$

Let $\rho\in(0,1) $ be a given number. Consider the
initial-boundary value problem
\begin{equation}\label{eq}
D_t^\rho u(x,t) + A(D)u(x,t) = 0, \quad x\in \mathbb{R}^N, \quad
0<t\leq T,
\end{equation}
\begin{equation}\label{bo}
\lim\limits_{|x|\rightarrow\infty} D^\alpha u(x,t)=0,\quad
|\alpha|\leq l-1, \quad 0<t\leq T,
\end{equation}
\begin{equation}\label{in}
u(x,0) = \varphi(x), \quad x\in \mathbb{R}^N,
\end{equation}
where $\varphi(x)$ is a given continuous function.

We call problem (\ref{eq}) - (\ref{in}) \emph{the forward
problem}.

\begin{defin}\label{def} A function $u(x,t)$ with the properties
$$D_t^\rho u(x,t)\,\, \text{and}\,\, A(D)u(x,t)\in C(\mathbb{R}^N\times
(0, T])$$ and satisfying conditions (\ref{eq}) - (\ref{in})  is
called the classical solution (or simply, solution) of the forward
problem.
\end{defin}

Denoting the Sobolev classes by $L^{\tau}_2(\mathbb{R}^N)$ (see
the definition in the next section), we can state an existence
theorem for this problem.

\begin{thm}\label{tfp} Let $\tau > \frac{N}{2}$ and $\varphi\in L^{\tau}_2(\mathbb{R}^N)$. Then the forward problem has a solution
in the form
\begin{equation}\label{fp}
u(x,t)=\int\limits_{\mathbb{R}^N} E_{\rho}(-A(\xi) t^\rho)\,
\hat{\varphi}(\xi)e^{ix\xi} d\xi.
\end{equation}
The integral uniformly converges with respect to $x\in
\mathbb{R}^N$ and for each $t\in (0, T]$, where
$\hat{\varphi}(\xi)$ is the Fourier transform of $\varphi$.

\end{thm}

If the solution of the forward problem $u(x,t)\in
L_2(\mathbb{R}^N)$, $t\in (0, T]$, then we may define the Fourier
transform
\[
\hat{u}(\xi, t)=(2\pi)^{-N}\int\limits_{\mathbb{R}^N}u(x, t)
e^{-ix\xi} dx.
\]

The corresponding uniqueness theorem has the form.
\begin{thm}\label{uniq} Let the following conditions be satisfied
for  all $t\in (0, T]$
\begin{enumerate}
\item
 $\varphi \in C(\mathbb{R}^N)$,
\item $\lim\limits_{|x|\rightarrow\infty} D^\alpha u(x,t)=0,\quad
l\leq|\alpha|\leq m-1$,
\item $D^\alpha u(x,t)\in L_2(\mathbb{R}^N),\quad |\alpha|\leq m,$
\item $\hat{u}(\xi, t)\in L_1(\mathbb{R}^N)$.
\end{enumerate}
Then there can be only one solution to the forward problem.

\end{thm}

\begin{rem}\label{Runiq}We will prove that under the condition of Theorem \ref{tfp} for initial function $\varphi$,
all four conditions of Theorem \ref{uniq} are also satisfied.
Thus, if we add these four conditions to Definition \ref{def},
then Theorem \ref{tfp} guarantees both the existence and the
uniqueness of such a solution.

\end{rem}

In recent years, numerous works of specialists have appeared,
where they study various initial-boundary value problems for
various subdiffusion equations. Let us mention only some of these
works. Basically, the case of one spatial variable $x\in
\mathbb{R}$ and subdiffusion equation with "the elliptical part"
$u_{xx}$ were considered (see, for example, handbook Machado,
aditor \cite{Mach}, book of A.A. Kilbas et al. \cite{Kil} and
monograph of A. V. Pskhu \cite{PSK}, and references in these
works). The paper Gorenflo, Luchko and Yamamoto \cite{GorLuchYam}
is devoted to the study of subdiffusion equations in Sobelev
spaces. In the paper by Kubica and Yamamoto \cite{KubYam},
initial-boundary value problems for equations with time-dependent
coefficients are considered. In multidimensional case ($x\in
\mathbb{R}^N$) instead of the differential expression $u_{xx}$
authors considered either the Laplace operator (\cite{Kil},
\cite{Agr} - \cite{PS1}), or pseudo-differential operators with
constant coefficients  in the whole space $\mathbb{R}^N$ (Umarov
\cite{SU}). In the last paper the initial function $\varphi\in
L_p(\mathbb{R}^N)$ is such, that the Fourier transform
$\hat{\varphi}$  is compactly supported. The authors of the recent
paper \cite{AO} considered initial-boundary value problems for
subdiffusion equations with arbitrary elliptic differential
operators in bounded domains.

\textbf{2.} Determining the correct order of an equation in
applied fractional modeling plays an important role. The
corresponding inverse problem for subdiffusion equations has been
considered by a number of authors (see a survey paper Li, Liu and
Yamamoto \cite{LiLiu} and references therein, \cite{Che}
-\cite{AA}). Note that in all known works the subdiffusion
equation was considered in a bounded domain $\Omega\subset
\mathbb{R}^N$. In addition, it should be noted that in
publications \cite{Che} -\cite{Jan} the following relation was
taken as an additional condition
\begin{equation}\label{ex1}
u(x_0,t)= h(t), \,\, 0<t<T,
\end{equation}
at a monitoring point $x_0\in \overline{\Omega}$. But this
condition, as a rule (an exception is work \cite{Jan} by J. Janno,
where both the uniqueness and the  existence are proved), can
ensure only the uniqueness of the solution of the inverse problem
\cite{Che} - \cite{LiL}. Authors of paper Ashurov and Umarov
\cite{AU} considered as an additional information the value of
projection of the solution onto the first eigenfunction of the
elliptic part of subdiffusion equation. Note, results of paper
\cite{AU} are applicable only in case, when the first eigenvalue
is equal to zero. The uniqueness and existence of an unknown order
of the fractional derivative in the subdiffusion equation were
proved in the recent work of Alimov and Ashurov \cite{AA}. In this
case, the additional condition is $ || u (x, t_0) || ^ 2 = d_0 $,
and the boundary condition is not necessarily homogeneous.

Now let us consider the order of fractional derivative $\rho$ in
equation (\ref{eq}) as an unknown parameter. We formulate our
inverse problem in the following way. Let us fix a vector
$\xi_0\neq 0$, such that $\hat{\varphi}(\xi_0)\neq 0$ and put
$\lambda_0=A(\xi_0)>0$. To determine the order $\rho$ of the
fractional derivative in (\ref{eq}) we use the following extra
data:
\begin{equation}\label{exN}
U(t_0, \rho) \equiv |\hat{u}(\xi_0, t_0)|\ =\ d_0,
\end{equation}
where $t_0>0$ is a fixed time instant. Obviously, Fourier
transform $\hat{u}$ of the solution depends on parameter $\rho$.

Problem (\ref{eq}) - (\ref{in}) together with extra condition
(\ref{exN}) is called \emph{the inverse problem}.

\begin{defin}\label{Idef} A pair $\{u(x,t), \rho\}$ of the solution $u(x,t)$ to the forward problem
and the parameter  $\rho\in (0,1)$ is called a classical solution
(or simply, solution) of the inverse problem.

\end{defin}

Let us denote by $E_{\rho}(t)$ the Mittag-Leffler function of the
form
$$
E_{\rho}(t)= \sum\limits_{k=0}^\infty \frac{t^k}{\Gamma(\rho
k+1)}.
$$

To solve the inverse problem fix the number $\rho_0\in (0,1)$ and
consider the problem for $\rho\in [\rho_0, 1)$.

\begin{lem}\label{W} For $ \rho_0 $ from the interval $ 0 <\rho_0 <1 $,
there is a number $ T_0 = T_0 (\lambda_0, \rho_0) $ such that for
all $ t_0 \geq T_0 $ and for arbitrary $ \varphi \in L_2
(\mathbb{R}^N) $ the function $ U (t_0, \rho) $ decreases
monotonically with respect to $ \rho \in [\rho_0,1] $.

\end{lem}

\begin{rem}\label{T0} The number $T_0(\lambda_0, \rho_0)$ can be chosen as
\[
T_0= e^k, \quad k \geq \frac{1}{\rho_0} \max
\bigg\{\frac{B_1}{B_2}, \ln\frac{2B_2 k}{\lambda_0}\bigg\},
\]
where $B_1=43$ and $B_2=4.6$. But if $ \rho_0 \lambda_0> 0.0075 $,
then you can just put $ T_0 = 2 $.
\end{rem}

The result related to the inverse problem has the form.

\begin{thm}\label{tin} Let $\varphi\in L^{\tau}_2(\mathbb{R}^N)$,  $\tau > \frac{N}{2}$, and  $t_0\geq T_0$.
Then the inverse problem  has a unique solution $\{u(x,t), \rho
\}$ if and only if
\begin{equation}\label{thm2}
e^{-\lambda_0}< \frac{d_0}{|\hat{\varphi}(\xi_0)|}\leq
E_{\rho_0}(-\lambda_0 t_0^{\rho_0}).
\end{equation}
\end{thm}

\textbf{3.} Finally, we will consider another inverse problem of
determining both the orders of fractional derivatives with respect
to time and the spatial derivatives in the subdiffusion equations.

For the best of our knowledge, only in the following two papers
\cite{Tar} and \cite{Yama} such  inverse problems were studied and
only the uniqueness theorems ware proved  (note that uniqueness is
a very important property of a solution from an application point
of view). In paper \cite{Tar} by Tatar and Ulusoy it is considered
the initial-boundary value problem for differential equation
\[
\partial_t^\rho u(t,x) = - (-\triangle)^\sigma u(t,x), \quad t>0, \ x \in
(0,1),
\]
where $\triangle^\sigma$ is the one-dimensional fractional Laplace
operator, $\rho \in (0,1)$ and $\sigma \in (1/4,1)$. The authors
have proved that if the initial function $ \varphi (x) $ is
sufficiently smooth and all its Fourier coefficients are positive,
then the two-parameter inverse problem  with additional
information (\ref{ex1}) has a unique solution. As for physical
backgrounds for two-parameter differential equations, see, for
example, \cite{Meer}.

In \cite{Yama}, M. Yamamoto proved the uniqueness theorem for the
above two-parameter inverse problem in $N$-dimensional bounded
domain $\Omega$ with smooth boundary $\partial \Omega$. The
conditions on the initial function found in this work are less
restrictive, for example, if $\varphi$ is zero on $\partial
\Omega$, $\varphi \in L_2^\tau(\Omega)$, $\tau>N/2$, $\varphi\geq
0$ in $\Omega$ and $\varphi(x_0)\neq 0$, then the uniqueness
theorem is true.

Let us denote by $A$ an operator in $L_2(\mathbb{R}^N)$ with
domain of definition $D(A)=C_0^\infty(\mathbb{R}^N)$, acting as
$Af(x)=A(D)f(x)$. It is easy to verify that the closure $ \hat{A}
$ of operator $ A $ is positive and selfadjoint. Therefore, by
virtue of the von Neumann theorem, for any $ \sigma> 0 $, we can
introduce the degree of the operator $ \hat{A} $ as
$$
\hat{A}^\sigma f(x)=\int\limits_0^\infty \lambda^\sigma d
P_\lambda f (x)=\int\limits_{\mathbb{R}^N} A^\sigma(\xi)
\hat{f}(\xi) e^{ix\xi}d\xi,
$$
where projectors $P_\lambda$ defined as
\[
P_\lambda f (x)=\int\limits_{A(\xi)<\lambda}\hat{f}(\xi) e^{ix\xi}
d\xi.
\]
The domain of definition of this operator is determined from the
condition $\hat{A}^\sigma f(x)\in L_2(\mathbb{R}^N)$ and has the
form
$$
D(\hat{A}^\sigma)=\{f\in L_2(\mathbb{R}^N):
\int\limits_{\mathbb{R}^N} A^{2\sigma}(\xi) |\hat{f}(\xi)|^2 d\xi<
\infty\}.
$$

Suppose first that $\rho\in (\rho_0, 1)$ and $\sigma \in (0, 1)$
are given numbers and consider the initial-boundary value
(\emph{the second forward}) problem
\begin{equation}\label{eq2}
D_t^\rho v(x,t) + A^\sigma v(x,t) = 0, \quad x\in \mathbb{R}^N,
\quad 0<t\leq T,
\end{equation}
\begin{equation}\label{bo2}
\lim\limits_{|x|\rightarrow\infty} v(x,t)=0,\quad  0<t\leq T,
\end{equation}
\begin{equation}\label{in2}
v(x,0) = \varphi(x), \quad x\in \mathbb{R}^N,
\end{equation}
where $\varphi(x)$ is a given continuous function.

The solution to this problem is defined similarly to the solution
to problem (\ref{eq}) - (\ref{in}) (see Definition \ref{def}). In
exactly the same way as Theorem \ref{tfp}, it is proved that if a
$\varphi\in L^{\tau}_2(\mathbb{R}^N)$ and $\tau > \frac{N}{2}$,
then the solution of the second forward problem has the form
\begin{equation}\label{fp2}
v(x,t)=\int\limits_{\mathbb{R}^N} E_{\rho}(-A^\sigma(\xi)
t^\rho)\, \hat{\varphi}(\xi)e^{ix\xi} d\xi,
\end{equation}
where the integral uniformly converges in $x\in \mathbb{R}^N$ and
for each $t\in (0, T]$.

Now suppose, that parameters $\rho$ and $\sigma$ are unknown. To
find these numbers one obviously needs two extra conditions. It
should be noted, that the proposed method, for simultaneously
finding both the order of fractional differentiation $\rho$ and
the power $\sigma$ is applicable if there exists such
$\xi_0\in\mathbb{R}^N$, so that $A(\xi_0)=1$ and
$\hat{\varphi}(\xi_0)\neq 0$. Let $\xi_0$ be one of such a vector.
We consider the following information as additional conditions:
\begin{equation}\label{exrho}
V(\xi_0, t_0, \rho, \sigma)=|\hat{v}(\xi_0, t_0)|=d_0, \quad
t_0\geq T_0(1, \rho_0),
\end{equation}
\begin{equation}\label{exsigma}
V(\xi_1, t_1, \rho, \sigma)=|\hat{v}(\xi_1, t_1)|=d_1, \quad
A(\xi_1)=\lambda_1 (\neq 1)\geq \Lambda_1 (\rho, \sigma_0), \quad
t_1>0,
\end{equation}
where $\xi_1$ is such that $\hat{\varphi}(\xi_1)\neq 0$ and
$\Lambda_1$ is defined in (\ref{L1}).

We call problem (\ref{eq2}) - (\ref{in2}) together with extra
conditions (\ref{exrho}) and (\ref{exsigma}) \emph{the second
inverse problem}.

Note that $V(\xi_0, t_0, \rho, \sigma)$ is actually independent of
$\sigma$. Therefore, to solve the second inverse problem, we first
find the unique $\rho^\star$ that satisfies the relation
(\ref{exrho}). Then, assuming that $\rho^\star$ is already known,
using relation (\ref{exsigma}), we find the second unknown
parameter $\sigma^\star$.

\begin{thm}\label{tin2} Let $\varphi\in L^{\tau}_2(\mathbb{R}^N)$ and  $\tau > \frac{N}{2}$.
Then there exists unique $\rho^\star$, satisfying (\ref{exrho}),
if and only if  $d_0$ satisfies inequalities (\ref{thm2}) with
$\lambda_0 =1$. For $\sigma^\star$ to exist, it is necessary and
sufficient that $d_1$ satisfies the inequalities
\begin{equation}\label{thm3}
E_{\rho^\star}(-\lambda_1
t_0^{\rho^\star})<\frac{d_1}{|\hat{\varphi}(\xi_1)|}
<E_{\rho^\star}(-\lambda_1^{\sigma_0} t_0^{\rho^\star}).
\end{equation}
\end{thm}

It should also be noted that the theory and applications of
various inverse problems, on determining the coefficients of the
equation, the right-hand side, and also on determining the initial
or boundary functions for differential equations of integer order
are discussed in Kabanikhin \cite{Kab} (see also references
therein).

\section{Forward problems}
In the present section we prove Theorems \ref{tfp} and \ref{uniq}
and equation (\ref{fp2}).

The class of functions $L_2(\mathbb{R}^N)$ which for a given fixed
number $a> 0$ make the norm
\[
||f||^2_{L_2^a(\mathbb{R}^N)}=\big|\big|\int\limits_{\mathbb{R}^N}(1+|\xi|^2)^{\frac{a}{2}}\hat{f}(\xi)
e^{ix\xi}\big|\big|^2_{L_2(\mathbb{R}^N)}=\int\limits_{\mathbb{R}^N}(1+|\xi|^2)^a|\hat{f}(\xi)|^2
d\xi
\]
finite is termed the Sobolev class $L_2^a(\mathbb{R}^N)$. Since
for $\tau>0$ and some constants $c_1$ and $c_2$ one has the
inequality
\begin{equation}\label{A}
c_1(1+|\xi|^2)^{\tau m} \leq 1+A^{2\tau}(\xi)\leq c_2
(1+|\xi|^2)^{\tau m},
\end{equation}
then $ D(\hat{A}^\tau)= L_2^{\tau m}(\mathbb{R}^N)$.

Let $I$ be the identity operator in $L_2(\mathbb{R}^N).$ Operator
$(\hat{A}+I)^\tau$ is defined in the same way as operator
$\hat{A}^\sigma$.

\textbf{Proof of Theorem \ref{tfp}} is based on the following
lemma (see M.A. Krasnoselski et al. \cite{Kra}, p. 453), which is
a simple consequence of the Sobolev embedding theorem.

\begin{lem}\label{CL} Let $\nu > 1+\frac{N}{2m}$. Then for any $|\alpha|\leq m$
operator $D^\alpha (\hat{A}+I)^{-\nu}$ continuously maps from
$L_2(\mathbb{R}^N)$ into $C(\mathbb{R}^N)$ and moreover the
following estimate holds true
\begin{equation}\label{CL1}
||D^\alpha (\hat{A}+I)^{-\nu} f||_{C(\mathbb{R}^N)} \leq C
||f||_{L_2(\mathbb{R}^N)}.
\end{equation}

\end{lem}
\begin{proof}For any
$a>N/2$ one has  the Sobolev embedding theorem:
$L^a_2(\mathbb{R}^N)\rightarrow C(\mathbb{R}^N)$, that is
$$||D^\alpha (\hat{A}+I)^{-\nu} f||_{C(\mathbb{R}^N)} \leq
C||D^\alpha (\hat{A}+I)^{-\nu} f||_{L_2^a(\mathbb{R}^N)}.
$$
Therefore, it is sufficient to prove the inequality
$$
||D^\alpha (\hat{A}+I)^{-\nu} f||_{L_2^a(\mathbb{R}^N)} \leq
C||f||_{L_2(\mathbb{R}^N)}.
$$
But this is a consequence of the
estimate
$$
\int\limits_{\mathbb{R}^N}|\hat{f}(\xi)|^2|\xi|^{2|\alpha|}(1+A(\xi))^{-2\nu}(1+|\xi|^2)^a
d\xi\leq C\int\limits_{\mathbb{R}^N}|\hat{f}(\xi)|^2 d\xi,
$$
that is valid for $\frac{N}{2}< a \leq \nu m - |\alpha|$.

\end{proof}

To prove the existence of the forward problem's solution we remind
the following estimate of the Mittag-Leffler function with a
negative argument  (see, for example, \cite{Gor},  p.29)
\begin{equation}\label{M}
|E_{\rho}(-t)|\leq \frac{C}{1+ t}, \quad t>0.
\end{equation}

In accordance with Definition \ref{def}, we will first show that
for function (\ref{fp}) one has $A(D)u(x, t)\in
C(\mathbb{R}^N\times (0, T])$ (that is one can validly apply the
operators $D^\alpha$, $|\alpha|\leq m$, to the series in
(\ref{fp}) term-by-term).

Consider the truncated integral
\begin{equation}\label{S}
S_\mu(x, t)=\int\limits_{0}^\mu E_{\rho}(-\lambda t^\rho)d
P_\lambda \varphi(x)=\int\limits_{A(\xi)<\mu} E_{\rho}(-A(\xi)
t^\rho)\, \hat{\varphi}(\xi)e^{ix\xi} d\xi.
\end{equation}
Let $\tau>\frac{N}{2}$ and $\nu=1+\frac{\tau}{m}> 1+\frac{N}{2m}$.
Then
\[
S_\mu(x, t)=(\hat{A}+I)^{-\tau/m -1}\int\limits_{0}^\mu
(\lambda+1)^{\tau/m+1} E_{\rho}(-\lambda t^\rho)d P_\lambda
\varphi(x).
\]
Therefore by virtue of Lemma \ref{CL} one has
$$
||D^\alpha S_\mu(x, t)||^2_{C(\mathbb{R}^N)}=||D^\alpha
(\hat{A}+I)^{-\tau/m-1}\int\limits_{0}^\mu (\lambda+1)^{\tau/m+1}
E_{\rho}(-\lambda t^\rho)d P_\lambda
\varphi(x)||^2_{C(\mathbb{R}^N)}\leq
$$
\[
\leq C ||\int\limits_{0}^\mu (\lambda+1)^{\tau/m+1}
E_{\rho}(-\lambda t^\rho)d P_\lambda
\varphi(x)||_{L_2(\mathbb{R}^N)}.
\]
Using the Parseval equality, we will have
\[
||D^\alpha S_\mu(x,t)||^2_{C(\mathbb{R}^N)}\leq C
\int\limits_{A(\xi)<\mu} \big|(A(\xi)+1)^{\tau/m+1}
E_{\rho}(-A(\xi) t^\rho)\hat{\varphi}(\xi)\big|^2 d\xi.
\]
Applying the inequality (\ref{M}) gives $|(A(\xi)+1)
E_{\rho}(-A(\xi) t^\rho)|\leq C t^{-\rho}$. Therefore,
$$
||D^\alpha S_\mu(x,t)||^2_{C(\mathbb{R}^N)}\leq  C
t^{-2\rho}\int\limits_{A(\xi)<\mu} \big|(A(\xi)+1)^{\tau/m}
\hat{\varphi}(\xi)\big|^2 d\xi\leq
Ct^{-2\rho}||\varphi||^2_{L^\tau_2(\mathbb{R}^N)}.
$$
This implies the uniform in $x\in\mathbb{R}^N$ convergence of the
differentiated sum (\ref{S}) in the variables $ x_j $ for each $ t
\in (0, T]$.

Further, from equation (\ref{eq}) one has $D_t^\rho S_\mu(x,  t)=
- A(D)S_\mu(x,t)$. Therefore, proceeding  the above reasoning, we
arrive at $D_t^\rho u(x, t)\in C(\mathbb{R}^N\times(0. T])$.

It is not difficult to verify that equation (\ref{eq}) and the
initial condition (\ref{in}) are satisfied  (see, for example,
\cite{Gor}, page 173 and \cite{ACT}).

Let us show that the inclusion $\varphi\in L^\tau_2(\mathbb{R}^N),
\tau>N/2$,  implies $\hat{\varphi}\in L_1(\mathbb{R}^N)$. Indeed,
\[
\int\limits_{\mathbb{R}^N}|\hat{\varphi}(\xi)|
d\xi=\int\limits_{\mathbb{R}^N}|\hat{\varphi}(\xi)|(1+|\xi|^2)^{\tau/2}(1+|\xi|^2)^{-\tau/2}
d\xi\leq  C_\tau||\varphi||^2_{L^\tau_2(\mathbb{R}^N)}.
\]
Therefore, by virtue of inequality (\ref{M}), one has $
E_{\rho}(-A(\xi) t^\rho)\hat{\varphi}(\xi)\in L_1(\mathbb{R}^N)$.
Similarly, inequalities  (\ref{A}) and (\ref{M}) imply
$|\xi^\alpha E_{\rho}(-A(\xi) t^\rho)\hat{\varphi}(\xi)|\leq
C|A(\xi)E_{\rho}(-A(\xi) t^\rho)\hat{\varphi}(\xi)| \in
L_1(\mathbb{R}^N)$ for all $|\alpha|\leq m$. Hence, $D^\alpha
u(x,t)$, as a function of $x$, is the Fourier transform of a
$L_1$- function. Obviously, this implies  both (\ref{bo}) and
condition (2) of Theorem \ref{uniq}.

Thus Theorem \ref{tfp} and condition (2) of Theorem \ref{uniq} are
proved.

Consider the other three conditions of Theorem \ref{uniq}. The
inclusion $D^\alpha u(x, t)\in L_2(\mathbb{R}^N),\,|\alpha|\leq
m,\, \text{ for all} \quad t\in (0, T]$, is a consequence of
condition $\varphi \in L_2(\mathbb{R}^N)$. Indeed, using
inequalities (\ref{A}) and (\ref{M}) we arrive at
\[
||D^\alpha u(x, t) S_\mu(x,
t)||^2_{L_2(\mathbb{R}^N)}=\int\limits_{A(\xi)<\mu}
\big|\xi^\alpha E_{\rho}(-A(\xi) t^\rho)\hat{\varphi}(\xi)\big|^2
d\xi\leq
\]
\[
\leq C\int\limits_{A(\xi)<\mu} \big|A(\xi) E_{\rho}(-A(\xi)
t^\rho)\hat{\varphi}(\xi)\big|^2 d\xi\leq C
t^{-2\rho}||\varphi||^2_{L_2(\mathbb{R}^N)}.
\]
The property of function $\varphi\,$: $\hat{\varphi} \in
L_1(\mathbb{R}^N)$, established above, implies condition (4):
\[
|\hat{u}(\xi, t)|=|(2\pi)^{-N}E_{\rho}(-A(\xi)
t^\rho)\hat{\varphi}(\xi)|\leq C|\hat{\varphi}(\xi)|\in
L_1(\mathbb{R}^N).
\]
As regards condition (1) of Theorem \ref{uniq}, it is a direct
consequence of the Sobolev embedding theorem and the condition
$\varphi\in L^{\tau}_2(\mathbb{R}^N)$, $\tau
> \frac{N}{2}$, of Theorem \ref{tfp}.

\textbf{Proof of Theorem \ref{uniq}.} Let conditions (1)--(4) of
Theorem \ref{uniq} are satisfied. Observe, as it was shown above,
Theorem \ref{tfp} guarantee the fulfilment of these conditions.

Suppose that problem (\ref{eq})--(\ref{in}) has two solutions
$u_1(x, t)$ and $u_2(x, t)$. Our aim is to prove that $u(x,
t)=u_1(x, t)-u_2(x, t)\equiv 0$. Since the problem is linear, then
we have the following homogenous problem for $u(x, t)$:
\begin{equation}\label{eq1}
D_t^\rho u(x, t) + A(D)u(x, t) = 0,\quad x\in \mathbb{R}^N, \quad
0<t\leq T;
\end{equation}
\begin{equation}\label{bo1}
\lim\limits_{|x|\rightarrow \infty} D^\alpha u(x,t)=0,\,
|\alpha|\leq l-1, \quad 0<t\leq T;
\end{equation}
\begin{equation}\label{in1}
u(x, 0) = 0, \quad x\in \mathbb{R}^N.
\end{equation}

Let $u(x, t)$ be a solution of problem (\ref{eq1})--(\ref{in1}).
Since $u(x,t)\in L_2(\mathbb{R}^N)$, $t\in (0, T]$ (see condition
(3) of Theorem \ref{uniq}), we may define the Fourier transform
$\hat{u}(\xi, t)$ and according to condition (4) one has
$\hat{u}(\xi, t)\in L_1(\mathbb{R}^N)$. Therefore, by virtue of
Fubini's theorem, the following function of $t$ exists for almost
all $\lambda$:
\begin{equation}\label{w}
w_\lambda(t)=\int\limits_{A(\xi)=\lambda}\hat{u}(\xi, t)
d\sigma_\lambda(\xi),
\end{equation}
where $d\sigma_\lambda(\xi)$ is the corresponding surface element.

Since $u(x,t)$ is a solution of equation (\ref{eq1}), then (note,
$A(D) u(x, t)\in L_2(\mathbb{R}^N)$)
\[
D_t^\rho
w_\lambda(t)=-(2\pi)^{-N}\int\limits_{A(\xi)=\lambda}\int\limits_{\mathbb{R}^N}A(D)
u(x, t) e^{-ix\xi} dx \, d\sigma_\lambda(\xi).
\]
The inner integral exists as the Fourier transform of
$L_2$-function. In this integral, we integrate by parts. We will
take into account the following: $A(D)$ is a homogeneous symmetric
and even order differential expression; conditions (2) of Theorem
\ref{uniq}; and (\ref{bo1}). Then
\[
D_t^\rho
w_\lambda(t)=-(2\pi)^{-N}\int\limits_{A(\xi)=\lambda}\int\limits_{\mathbb{R}^N}A(-i\xi)
u(x, t) e^{-ix\xi} dx \, d\sigma_\lambda(\xi)= - \lambda
w_\lambda(t).
\]
Therefore, we have the following Cauchy problem for
$w_\lambda(t)$:
$$
D_t^\rho w_\lambda(t) +\lambda w_\lambda(t)=0,\quad t>0; \quad
w_\lambda(0)=0.
$$
This problem has the unique solution; hence, the function defined
by (\ref{w}), is identically zero (see, for example, \cite{Gor},
p. 173 and \cite{ACT}): $w_\lambda(t)\equiv 0$ for almost all
$\lambda>0$. Integrating the equation (\ref{w}) with respect to
$\lambda$ over the domain $(0, +\infty)$ we obtain
\[
\int\limits_{\mathbb{R}^N} \hat{u}(\xi, t) d\xi = 0, \,\, t> 0.
\]
Therefore, $\hat{u}(\xi, t) =0$ for almost all $\xi$, or $u(x, t)
=0$ for almost all $x$ and since $u(x, t)$ continuous on $x$, then
$u(x, t) =0$ for all $x$ and $t$. Thus Theorem \ref{uniq} is
proved.

Formula (\ref{fp2}) for the solution of the second forward problem
is established in exactly the same way with formula (\ref{fp}).

\section{First inverse problem}

\begin{lem}\label{elambda}
Given $\rho_0$ from the interval $0<\rho_0< 1$, there exists a
number $T_0=T_0(\lambda_0, \rho_0)$, such that for all $t_0\geq
T_0$ and  $\lambda \geq \lambda_0$ function $e_\lambda(\rho) =
E_\rho(-\lambda t_0^\rho)$ is positive and monotonically
decreasing with respect to $\rho\in [\rho_0, 1)$ and
\[
e_\lambda(1) < e_\lambda(\rho) \leq e_\lambda(\rho_0).
\]

\end{lem}

\begin{proof} Let us denote by $\delta(1; \beta)$ a contour oriented by
non-decreasing $\arg \zeta$ consisting of the following parts: the
ray $\arg  \zeta = -\beta$ with $|\zeta|\geq 1$, the arc
$-\beta\leq \arg \zeta \leq \beta$, $|\zeta|=1$, and the ray $\arg
\zeta = \beta$, $|\zeta|\geq 1$. If $0<\beta <\pi$, then the
contour $\delta(1; \beta)$ divides the complex $\zeta$-plane into
two unbounded parts, namely $G^{(-)}(1;\beta)$ to the left of
$\delta(1; \beta)$ by orientation, and $G^{(+)}(1;\beta)$ to the
right of it. The contour $\delta(1; \beta)$ is called the Hankel
path.

Let $\beta = \frac{3\pi}{4}\rho$, $\rho\in [\rho_0, 1)$. Then by
the definition of this contour $\delta(1; \beta)$, we arrive at
(note, $-\lambda t_0^\rho\in G^{(-)}(1;\beta)$, see \cite{Gor}, p.
27)
\begin{equation}\label{Erho}
E_{\rho}(-\lambda t_0^\rho)= \frac{1}{\lambda t_0^\rho \Gamma
(1-\rho)}-\frac{1}{2\pi i \rho\lambda
t_0^\rho}\int\limits_{\delta(1;\beta)}\frac{e^{\zeta^{1/\rho}}\zeta}{\zeta+\lambda
t_0^\rho} d\zeta = f_1(\rho)+f_2(\rho).
\end{equation}

Let $\Psi(\rho)$ be the logarithmic derivative of the gamma
function $\Gamma(\rho)$ (for the definition and properties of
$\Psi$ see \cite{Bat}). Then $\Gamma'(\rho) = \Gamma (\rho)
\Psi(\rho)$, and therefore,
$$
f_1'(\rho)=-\frac{\ln t_0 - \Psi (1-\rho)}{\lambda t_0^\rho \Gamma
(1-\rho)}.
$$
Since
\[
\frac{1}{\Gamma(1-\rho)}=\frac{1-\rho}{\Gamma(2-\rho)}, \quad
\Psi(1-\rho)=\Psi(2-\rho)-\frac{1}{1-\rho},
\]
the function $f_1'(\rho)$  can be
represented as follows
\[
f_1'(\rho)=-\frac{1}{\lambda t_0^\rho} \frac{(1-\rho)[\ln t_0 -
\Psi (2-\rho)]+1}{ \Gamma (2-\rho)}.
\]
If $\gamma\approx 0,57722$ is the Euler-Mascheroni constant, then
$-\gamma <\Psi(2-\rho)< 1-\gamma$. By virtue of this estimate we
may write
\begin{equation}\label{f1}
-f_1'(\rho)\geq  \frac{(1-\rho)[\ln t_0 -
(1-\gamma)]+1}{\Gamma(2-\rho)\lambda t_0^\rho}\geq
\frac{1}{\lambda t_0^{\rho}},
\end{equation}
 provided $\ln t_0 > 1-\gamma$ or $t_0\geq 2$.

On the other hand one has
$$
f_2'(\rho)=\frac{1}{2\pi i \rho\lambda
t_0^\rho}\int\limits_{\delta(1;\beta)}\frac{e^{\zeta^{1/\rho}}\zeta\bigg[-\frac{1}{\rho^2}|\zeta|^{1/\rho}(\ln
|\zeta|+i\beta) -\ln t_0-\frac{\lambda t_0^\rho \ln t_0}{\zeta
+\lambda t_0^\rho}\bigg]}{\zeta+\lambda t_0^\rho} d\zeta.
$$
Note, since $\beta = \frac{3\pi}{4}\rho$ and $\rho_0\leq\rho< 1$,
then for a negative number $z<0$ the following inequality holds
\[
\min\limits_{\zeta\in\delta(1;\beta)}|\zeta-z|\geq |z|.
\]
Therefore,
$$
|f_2'(\rho)|\leq
\frac{1}{I}\int\limits_{\delta(1;\beta)}|e^{\zeta^{1/\rho}}||\zeta|\bigg[\frac{1}{\rho^2}|\zeta|^{1/\rho}(\ln
|\zeta|+\beta) +2\ln t_0\bigg] |d\zeta|=
$$
$$
\frac{1}{I\cdot\rho^2}\int\limits_{\delta(1;\beta)}|e^{\zeta^{1/\rho}}||\zeta|^{1/\rho+1}\ln
|\zeta||d\zeta| +
\frac{\beta}{I\cdot\rho^2}\int\limits_{\delta(1;\beta)}|e^{\zeta^{1/\rho}}||\zeta|^{1/\rho+1}|d\zeta|
+ \frac{2\ln
t_0}{I}\int\limits_{\delta(1;\beta)}|e^{\zeta^{1/\rho}}||\zeta||d\zeta|,
$$
where $I=2\pi \rho(\lambda t_0^\rho)^2$. Let us denote the last
three integrals by $J_j$, $j=1,2,3$, correspondingly.

\begin{lem}\label{I} Let $0<\rho\leq 1$ and $m\in \mathbb{N}$.
Then
\[
I(\rho)=\frac{1}{\rho}\int\limits_1^\infty
e^{-\frac{1}{2}s^{\frac{1}{\rho}}} s^{\frac{m}{\rho}+1} ds\leq
C_m.
\]
\end{lem}
\begin{proof} Set $r=s^{\frac{1}{\rho}}$. Then
\[
s=r^\rho, \quad ds = \rho r^{\rho-1} dr.
\]
Therefore,
\[
I(\rho)=\int\limits_1^\infty e^{-\frac{1}{2}r} r^{m-1+2\rho}
dr\leq \int\limits_1^\infty e^{-\frac{1}{2}r} r^{m+1} dr = C_m.
\]

\end{proof}

It is not hard to verify, that
\[
C_2=\frac{2^4\cdot 16}{e}\approx94.2,\quad C_1=\frac{2^3\cdot
5}{e}\approx 14.72, \quad C_0=\frac{2^2\cdot 2}{e}\approx 3.
\]

Consider the integral $J_1$. Due to the presence of $ \ln |\zeta|
$, the integrand $ J_1 $ is equal to $ 0 $ for $ |\zeta | = 1 $.
Moreover, on the rays $\arg \zeta =\pm \beta, \beta
=\frac{3\pi}{4}\rho$, one has
\[
|e^{\zeta^{\frac{1}{\rho}}}|=\exp\big(\cos
\frac{\beta}{\rho}|\zeta|^{\frac{1}{\rho}}\big)=e^{-\frac{1}{2}|\zeta|^{\frac{1}{\rho}}}.
\]
Hence (note $\ln
|\zeta|^{\frac{1}{\rho}}<|\zeta|^{\frac{1}{\rho}}$) by virtue of
Lemma \ref{I},
$$
J_1=\frac{1}{I\cdot\rho}\int\limits_{\delta(1;\beta)}|e^{\zeta^{1/\rho}}||\zeta|^{1/\rho+1}\ln
|\zeta|^{\frac{1}{\rho}}|d\zeta| =
\frac{2}{I\cdot\rho}\int\limits_{1}^\infty
e^{-\frac{1}{2}|\zeta|^{\frac{1}{\rho}}}|\zeta|^{1/\rho+1}\ln
|\zeta|^{\frac{1}{\rho}}|d\zeta|\leq
$$
$$
\leq\frac{2}{I}\int\limits_{1}^\infty
e^{-\frac{1}{2}s^{\frac{1}{\rho}}} s^{\frac{2}{\rho}+1} ds\leq
\frac{C_2}{\pi\rho(\lambda t_0^\rho)^2}.
$$

The integrands in $ J_2 $ and $ J_3 $ do not vanish on the sphere
$\{|\zeta|=1\}$ and the measure of the corresponding arc
$-\beta\leq \arg \zeta \leq \beta$, $|\zeta|=1$, is equal to
$2\beta$. Therefore, using the same technique as above,  we obtain
$$
J_2=\frac{\beta}{I\cdot\rho^2}\int\limits_{\delta(1;\beta)}|e^{\zeta^{1/\rho}}||\zeta|^{1/\rho+1}|d\zeta|=
\frac{2\beta}{I\cdot\rho^2}\bigg[\int\limits_1^\infty
e^{-\frac{1}{2}s^{\frac{1}{\rho}}}
s^{\frac{1}{\rho}+1}ds+\beta\bigg]\leq
$$
$$
\leq
\frac{\frac{3}{2}\pi\rho(C_1\rho+\frac{3}{4}\pi\rho)}{2\pi\rho^3(\lambda
t_0^\rho)^2}=\frac{3(C_1+\frac{3}{4}\pi)}{4\rho(\lambda
t_0^\rho)^2}.
$$
Similarly,
\[
J_3=\frac{2\ln
t_0}{I}\int\limits_{\delta(1;\beta)}|e^{\zeta^{1/\rho}}||\zeta||d\zeta|=\frac{4\ln
t_0}{I}\bigg[\int\limits_1^\infty
e^{-\frac{1}{2}s^{\frac{1}{\rho}}} s ds+\beta\bigg]\leq\frac{2 \ln
t_0(C_0+\frac{4}{3}\pi)}{\pi(\lambda t_0^\rho)^2}
\]

Thus we have
\[
|f'_2(\rho)|\leq\frac{B_1/\rho+B_2 \ln t_0}{(\lambda t_0^\rho)^2},
\]
where $B_1=43$ and $B_2=4.6$.

Therefore, taking into account estimate (\ref{f1}), we have
\begin{equation}\label{der}
\frac{d}{d\rho} e_\lambda(\rho)< -\frac{1}{\lambda
t_0^{\rho}}+\frac{B_1/\rho+B_2 \ln t_0}{(\lambda t_0^{\rho})^2}.
\end{equation}
In other words, this derivative is negative if
\[
t_0^{\rho}>\frac{B_1/\rho+B_2 \ln t_0}{\lambda}
\]
for all $\rho\in [\rho_0, 1)$ or, which is the same,
\begin{equation}\label{t01}
t_0^{\rho_0}>\frac{B_1/{\rho_0}+B_2 \ln t_0}{\lambda}.
\end{equation}

Next, consider two cases: $B_1/{\rho_0}> B_2 \ln t_0$ and
$B_1/{\rho_0}\leq B_2 \ln t_0$. Recall that to satisfy inequality
(\ref{f1}) we assumed that $ t_0 \geq 2 $.

\textbf{Case 1.} Let $B_1/{\rho_0}> B_2 \ln t_0$. Then
$$
\ln t_0^{\rho_0} < \frac{B_1}{B_2}, \quad \text{or} \quad
t_0^{\rho_0}< e^{\frac{B_1}{B_2}}
$$
(the latter inequality is satisfied, say if $t_0\leq
e^{B_1/B_2}$). Therefore, inequality (\ref{t01}) is satisfied if
\[
t_0^{\rho_0}>\frac{2B_1}{\rho_0\lambda}.
\]

Thus, from the last two inequalities it follows that if
\begin{equation}\label{T01}
\rho_0\cdot \lambda_0>2B_1 \cdot e^{-\frac{B_1}{B_2}} \quad
\text{and} \quad 2\leq t_0\leq e^{\frac{B_1}{B_2}}
\end{equation}
then derivative (\ref{der}) is negative for all $\lambda\geq
\lambda_0$ and $\rho\in [\rho_0,1)$. Note $2B_1 \cdot
e^{-\frac{B_1}{B_2}}< 0.0075$ (see Remark \ref{T0}).

If $\rho_0$ and $\lambda_1$ are such small numbers, that the first
inequality of (\ref{T01}) does not hold true, then consider Case
2. In this case, $ t_0 $ should be chosen large enough.

\textbf{Case 2.} Let $B_1/{\rho_0}\leq B_2 \ln t_0$ or, which is
the same, $t_0^{\rho_0}\geq e^{\frac{B_1}{B_2}}$. From (\ref{t01})
one has
\[
t_0^{\rho_0}\geq \frac{2B_2 \ln t_0}{\lambda}.
\]
Thus, in Case 2 in order for the derivative (\ref{der}) to be
negative for all $\lambda\geq \lambda_0$ and $\rho\in [\rho_0,1)$,
it is sufficient that the following inequality takes place
$t_0\geq T_0$, where (see Remark \ref{T0})
\begin{equation}\label{T02}
T_0= e^k, \quad k \geq \frac{1}{\rho_0} \max
\bigg\{\frac{B_1}{B_2}, \ln\frac{2B_2 k}{\lambda_0}\bigg\}.
\end{equation}

Finally, by virtue of inequality $e_\lambda(1)=e^{-\lambda t}>0$,
one has $e_\lambda(\rho)>0$.

\end{proof}
Since
$$
U(t,\rho)=|\hat{u}(\xi_0, t)|=E_\rho(-A(\xi_0) t^\rho)
|\hat{\varphi}(\xi_0)|=E_\rho(-\lambda_0 t^\rho)
|\hat{\varphi}(\xi_0)|,
$$
Lemma \ref{W} follows immediately from Lemma \ref{elambda}.
Theorem \ref{tin} is an easy consequence of these two lemmas.

In conclusion, we make the following remark. If the elliptic
polynomial $A(\xi)$ is nonhomogeneous, that is
$A(\xi)=\sum\limits_{|\alpha|\leq m} a_\alpha \xi^\alpha$ and
moreover, $A(\xi)\geq \lambda_0>0$, then from Lemma \ref{elambda}
it follows:

\emph{If $t_0\geq T_0$ and $T_0$ is as above, then $E_\rho(-A(\xi)
t^\rho)$, as a function of $\rho$, is positive and decreases
monotonically in $\rho\in [\rho_0, 1]$ for any} $\xi\in
\mathbb{R}^N$.

Therefore, in this case you can also consider various options for
the function $U(t, \rho)$. Examples $U(t, \rho)=||Au(x, t)||^2$,
$U(t, \rho)=||u(x, t)||^2$, $U(t, \rho)=(u,\varphi)$.

\section{Second inverse problem}

To prove Theorem \ref{tin2}, we first find the unknown parameter $
\rho $. Suppose, as required by Theorem \ref{tin2}, that $d_0$
satisfies condition (\ref{thm2}) with $\lambda_0=A(\xi_0)=1$.
Then, as it follows from Lemma \ref{elambda}, for all $t_0\geq
T_0(1, \rho_0)$ the equation
\[
V(\xi_0, t_0, \rho, \sigma)=|\hat{v}(\xi_0,
t_0)|=E_\rho(-t^\rho)|\varphi(\xi_0)|=d_0
\]
has the unique solution $\rho^\star\in (\rho_0, 1)$.

Now let us define $\sigma^\star\in [\sigma_0, 1)$, which
corresponds to the already found $\rho^\star$ and satisfies
condition (\ref{exsigma}). Let $\beta = \frac{3\pi}{4}\rho^\star$.
Then formula (\ref{Erho}) will have the form
\begin{equation}\label{Elambda}
E_{\rho^\star}(-\lambda^\sigma t_0^{\rho^\star})=
\frac{1}{\lambda^\sigma t_0^{\rho^\star} \Gamma
(1-\rho^\star)}-\frac{1}{2\pi i \rho^\star\lambda^\sigma
t_0^{\rho^\star}}\int\limits_{\delta(1;\beta)}\frac{e^{\zeta^{1/{\rho^\star}}}\zeta}{\zeta+\lambda^\sigma
t_0^{\rho^\star}} d\zeta = g_1(\sigma)+g_2(\sigma).
\end{equation}
One has
\[
g'_1(\sigma)=-\frac{\ln \lambda}{\lambda^\sigma
t_0^{\rho^\star}\Gamma (1-\rho^\star)}
\]
and
\[
g'_2(\sigma)=\frac{(1+t_0^{\rho^\star})\ln \lambda}{2\pi i
\rho^\star\lambda^\sigma
t_0^{\rho^\star}}\int\limits_{\delta(1;\beta)}\frac{e^{\zeta^{1/{\rho^\star}}}\zeta}{\zeta+\lambda^\sigma
t_0^{\rho^\star}} d\zeta.
\]
It is not hard to verify, that $g'_2(\lambda)$ has the estimate
(is proved in a completely similar way to estimate of $J_3$)
\[
|g'_2(\sigma)|\leq \frac{(1+t_0^{\rho^\star})\ln \lambda}{\pi
(\lambda^\sigma t_0^{\rho^\star})^2} \big(C_0+\frac{4}{3}\pi\big)<
\frac{2.3(1+t_0^{\rho^\star})\ln \lambda}{(\lambda^\sigma
t_0^{\rho^\star})^2}.
\]
Therefore, for all $t_0>1$ we have
$$
\frac{d}{d\sigma} e_{\lambda^\sigma}(\rho^\star)<-\frac{\ln
\lambda}{\lambda^\sigma t_0^{\rho^\star}\Gamma
(1-\rho^\star)}+\frac{5\ln \lambda}{\lambda^{2\sigma}
t_0^{\rho^\star}}.
$$
Hence this derivative is negative if
\[
\lambda^\sigma\geq \lambda^{\sigma_0}\geq 5\,\Gamma
(1-\rho^\star).
\]

Thus, if $\lambda_1\geq \Lambda_1$, where (see Remark \ref{T0})
\begin{equation}\label{L1}
\Lambda_1= e^n, \quad n \geq \frac{\ln (5\,
\Gamma(1-\rho^\star))}{\sigma_0},
\end{equation}
then $e_{\lambda_1^\sigma}(\rho^\star)$, as a function of
$\sigma\in[\sigma_0, 1)$, strictly decreases for all $t_0>1$.

Moreover, for all $\lambda_1\geq \Lambda_1$ and
$\sigma\in[\sigma_0, 1)$ the following estimate is fulfilled
\[
e_{\lambda_1}(\rho^\star)<
e_{\lambda_1^\sigma}(\rho^\star)<e_1(\rho^\star).
\]
The last estimate shows that if $ d_1 $ satisfies condition
(\ref{thm3}), then, assuming $ \rho^\star $ has already been
found, we can easily determine the parameter $ \sigma^\star $ from
equality (\ref{exsigma}), that is, from
\[
e_{\lambda_1^\sigma}(\rho^\star)|\hat{\varphi}(\xi_1)|=d_1.
\]

\section{Acknowledgement}
The authors convey thanks to Sh. A. Alimov for discussions of
these results.

\bibliographystyle{amsplain}

\end{document}